\documentclass{amsart}

\newtheorem{theorem}{Theorem}[section]
\newtheorem{lemma}[theorem]{Lemma}
\newtheorem{corollary}[theorem]{Corollary}
\newtheorem{prop}[theorem]{Proposition}

\theoremstyle{definition}
\newtheorem{definition}[theorem]{Definition}

\newtheorem{question}{Question}

\theoremstyle{remark}

\numberwithin{equation}{section}

\begin{document}

 \title[]{RADICALLY FINITE RINGS AND SPACE CURVES}


\author[]{Vahap Erdoğdu}
\address{Department of Electrical and Electronic Engineering Piri Reis University, 34940 Tuzla, Istanbul, Turkey}
\curraddr{}
\email{verdogdu@pirireis.edu.tr}
\thanks{}






\begin{abstract}
We define radically finite rings and show that a finite dimensional radically finite ring is Noetherian and that if either $R$ is a finite character Hilbert domain that contains a field of characteristic zero or a finite dimensional Prüfer domain, then the polynomial ring $R[X]$ over $R$  is radically finite if and only if $R$ is a Dedekind  domain with torsion ideal class group. We then consider the radically finite condition on UFD and show that there does not exist a finite character 
 UFD $R$ of Krull dimension 2 over which  the polynomial ring $R[X]$  is radically finite. 
From this it follows that not all space curves are set theoretic complete intersection.
\end{abstract}

\maketitle
\section{Introduction, Definitions and Motivations}

We begin by recalling (see, e.g. \cite{4,5,6}) that an ideal $I$ of a commutative ring $R$ with identity is radically perfect if among the ideals of $R$ whose radical is equal to the radical of $I$ the one with the least number of generators has this number of generators equal to the height of $I$ and if  the  height of $I$ is zero, then the radical of $I$ is equal to the radical of a principal ideal generated by a zero divisor of $R.$ That is $I$ is radically perfect if
$ ht(I)=\inf \{n \,| \, \sqrt{I}=\sqrt{(\theta_1, \theta_2,...,\theta_n)},  \theta_i\in R \: \text{and if} \: ht(I)=0, \: \text{then} \: \sqrt{I}=\sqrt{(\theta)}\:
 \text{for some zero divisor} \: \theta \:  \text{of}\, \,  R\}.$

\begin{definition}
  Call a ring $R$ radically finite if each prime ideal $P$ of $R$ is radically perfect and the set of ideals of $R$ generated by $ht(P)$-number of elements with radical $P$ has a maximal member $A$ and that there are only finitely many ideals (in a chain) between $A$ and $P$.
\end{definition}
The need for this definition arose from the following:
\begin{question}
 Is a Prüfer domain of finite Krull dimension having each prime ideal of it radically perfect a Dedekind domain with torsion ideal class group?
\end{question}

The answer to this question is no. For if $(R,M)$ is a non-Noetherian valuation domain of Krull dimension one, then the maximal ideal $M$ of $R$ is radically perfect, since $M=\sqrt{(\theta)}$ for any non-zero element  $\theta$ of $M.$ But $R$ is not a Dedekind domain i.e. is not a discrete valuation ring (DVR). That is we have   $M=\sqrt{(\theta)}$ but $M\neq (\theta)$, and so there is an element $\Pi_1$ in $M$ not in $(\theta),$ and as $R$ is  a valuation ring $(\theta)\subset (\Pi_1)$ with $M=\sqrt{(\Pi_1)}$ but $M\neq (\Pi_1).$ Continuing in  this way we obtain an ascending chain
 \[ (\theta)\subset(\Pi_1)\subset(\Pi_2)\subset...\subset(\Pi_n)\subset...\]
of principal ideals between $(\theta)$ and $M.$ Now if $R$ is radically finite then this chain is stationary and so there is a least positive integer $n$ such that $(\Pi_n)=(\Pi_k)$ for all $k\geq n.$ Hence $M=(\Pi_n)$ and so $(R,M)$ is a DVR.

Note that any radically finite valuation ring is a DVR. For if $(R,M)$ is any valuation ring, then $M=\sqrt{(\theta)}$ for any element $\theta$ in $M$ not contained in the union of all prime ideals of $R$ properly contained in $M.$ So if $(R,M)$ is radically finite then $ht(M)=1=dim R,$ and so $(R,M)$ is a DVR.

\section{Finite Dimensional Radically Finite Rings}
It easily follows from the definition that a semilocal Noetherian ring of dimension $\leq1$ is radically finite and for the converse statement we have Corollary 2.2. below which is an easy consequence of the following result.
\begin{theorem}
  Let $R$ be a radically finite ring and let $P$ be a prime ideal of $R.$ If $P$ is of finite height, then $P$ is finitely generated.
\end{theorem}
\begin{proof}
 Let $P$ be of finite height and let $n=ht(P).$ Since $R$ is radically finite, there exists an $n$-generated ideal $A$ of $R$ such that
  $\sqrt{A}=P$ and $\{I\mid I \: \,   \text{is an ideal of} \, \: R \: \,\text{in } \\ \text{a chain between A and P} \}$ is finite. Observe that $P/A$ is an $R$-module of finite length, and
  hence $P/A$ is a finitely generated $R$-module. Since $A$ is a finitely generated ideal of $R,$ we infer that $P$ is finitely generated.
\end{proof}
\begin{corollary} Finite dimensional radically finite rings are Noetherian.
\end{corollary}
\begin{proof}
Let $R$ be a finite dimensional radically finite ring and $P$ be any prime ideal of $R.$ Then it follows from the above theorem that $P$ is finitely generated and hence
by Cohen's theorem $R$ is Noetherian (see, \cite{1}).
\end{proof}

\begin{corollary}
 A finite dimensional Prüfer domain is radically finite if and only if it is a Dedekind domain with torsion ideal class group.
\end{corollary}

\begin{proof}
  Let $R$ be a finite dimensional radically finite Prüfer domain. Then $R$ is a Noetherian Prüfer domain and hence is a Dedekind domain. Thus $R$ is  a  radically finite Dedekind domain, and so each non-zero prime ideal of $R$ is the radical of a principal ideal. That is if $M$ is a non-zero prime ideal of $R,$ then $M=\sqrt{(\Pi)}$ for some non-zero element $\Pi$ of $M.$ But then $M^{n}=(\Pi)$ for some positive integer $n$ and therefore the class group of $R$ is torsion.

The converse requires no comments.
\end{proof}

\begin{corollary}
  A finite dimensional B$\acute{e}$zout Domain is radically finite if and only if it is a principal ideal domain PID.
\end{corollary}

Note that a Dedekind domain need not be radically finite, unless its class group is torsion, and when that is the case, then the polynomial ring over it is also  radically finite. In fact we have:

\begin{corollary}\label{2.4}
  The polynomial ring $R[X]$ over a finite dimensional Prüfer domain $R$ is radically finite if and only if $R$ is a Dedekind domain with torsion ideal class group.
\end{corollary}

\begin{proof}
Since $R$ is Prüfer, dim$R[X]$=dim$R$+1 and so $R[X]$ is a finite dimensional radically finite ring and therefore is Noetherian. Hence $R$ is Noetherian and being Prüfer is a Dedekind domain. But then the result follows from Theorem 2.1 of \cite{3}.
\end{proof}

It now goes without saying that Artinian rings $\subset$ Finite dimensional radically finite rings $\subset$  Noetherian rings, and that radically finite Noetherian rings are just  Noetherian rings with radically perfect prime ideals. It is also clear from the above results that the following is a legitimate question.

\begin{question}
  Are finite dimensional radically finite rings Cohen-Macaulay?
\end{question}

The results of the next two sections are also in support of this question.

\section{Radically Finite Hilbert Rings}
We begin this section by recalling that a ring $R$ is of finite character if each non-zero element of $R$ is contained in only finitely many maximal ideals.\\

Our main result of this section is the following Theorem.
\begin{theorem}
Let $R$ be a finite character Hilbert domain containing a field of characteristic zero. Then the polynomial ring $R[X]$ over $R$ is radically finite if and only if $R$ is a Dedekind domain with torsion ideal class group.
\end{theorem}
To prove this theorem we need the following lemma.
\begin{lemma}
A finite character Hilbert domain is of Krull dimension one.
\end{lemma}
\begin{proof}
Let $R$ be a finite character Hilbert domain and $P$ be any height one prime ideal of $R$. Then $P$ is contained in only finitely many maximal ideals $M_1,M_2,...,M_n$ of $R$ and hence is a product of these maximal ideals. Thus $P=M_1M_2...M_n$ and so $P=M_i$ for some $1\leq i \leq n $ is a maximal ideal. That is each height one prime ideal of $R$ is a maximal ideal and therefore $R$ is of Krull dimension one.
\end{proof}
\begin{corollary}
Let $R$ be a finite character Hilbert domain. Then for any non-zero ideal $I$ of $R$ the polynomial ring $R/\sqrt{I}~[X]$ over the ring $R/\sqrt{I}$ is radically finite.
\end{corollary}
\begin{proof}
From the above proof we have $\sqrt{I}=M_1M_2...M_n$ for some maximal ideals $M_1,M_2,...,M_n$ of $R$. Hence $R/\sqrt{I}$ is a finite direct product of fields. But then $R/\sqrt{I}~[X]$ is a finite direct product of principal ideal domains and therefore itself is a principal ideal ring, and hence is radically finite.
\end{proof}

\begin{proof}\textit{of Theorem 3.1.} $R$ being a finite character Hilbert domain is of Krull dimension one which implies that $R[X]$ is of Krull dimension $\leq 3$. Thus, we have $R[X]$ a finite dimensional radically finite ring and therefore is Noetherian. But then the ring $R$ is Noetherian. That is $R$ is a Noetherian integral domain of Krull dimension one containing a field of characteristic zero over which the polynomial ring $R[X]$ has all its prime ideals are radically perfect. Hence the result now follows from Theorem 2.1 of [5].
\end{proof}

\section{Radically Finite UFD}
Here we present a not necessarily Noetherian UFD of Krull dimension 2 whose prime ideals are radically perfect and show that there does not exist a finite character radically finite UFD $R$ of Krull dimension 2 over which the polynomial ring $R[X]$ over $R$ is radically finite. From this  we deduce that not all space curves are set theoretic complete intersection.
 We note that there are non-Noetherian UFD of Krull dimension 2 (see, e.g. \cite{7}).
\begin{prop}\label{prop}
  Each prime ideal of a finite character UFD of Krull dimension 2 is radically perfect.
\end{prop}

\begin{proof}
  Let $R$ be such a domain and $P$ be any non-zero prime ideal of $R.$ Then we have $ht(P)=1$ or 2. If $ht(P)=1$ then $P$ is principal and so is radically perfect. So suppose that $ht(P)=2.$ (We now replace $P$ by $M$ for the sake of the application of this proof  in the next result). Thus $P=M$ is a maximal ideal of $R$ of height 2. Let $\Pi_1$ be any irreducible element of $M,$ then $\Pi_1$ is contained in only finitely many maximal ideals of $R,$ say $M,M_1,M_2,...,M_n$ are the only maximal ideals of $R$ containing  $\Pi_1.$ Next choose an element $\Pi_2$ in $M$ and $z$ in $\bigcap_{i=1}^{n}M_i$ such that $\Pi_2+z=1.$ Then it is clear that $M=\sqrt{(\Pi_1,\Pi_2)}.$  Therefore $M$ is radically perfect.
\end{proof}

\begin{definition}
Let $I$ be an ideal of $R.$ Then $j-$radical of $I$ is the intersection of all maximal ideals which contain $I$ and will be denoted by $j-\sqrt{I}.$
\end{definition}

\begin{theorem}\label{4.3}
  There does not exist a finite character UFD $R$ of Krull dimension 2 over which the polynomial ring $R[X]$ is radically finite.
\end{theorem}

\begin{proof}
Suppose on the contrary that such ring $R$ exists. Since $R$ is of dimension 2, $R[X]$ is of dimension $\leq 5$. Hence $R[X]$ is a finite dimensional radically finite ring and therefore is Noetherian and hence  $R$ is Noetherian, and so  $dim R[X]=3$. Let $M^{*}$ be any maximal ideal of $R[X]$ of height 3, and  $M=M^{*}\cap R$ be the contraction of $M^{*}$ in $R.$ Then $M$ is a height two prime ideal of $R$ i.e. a  maximal ideal of $R.$ This is because there can not be a chain of three distinct prime ideals in $R[X]$ contracting to the same prime ideal in $R$ [\cite{8}, Theorem 37]. Hence $M^{*}=(M,f)$ for some irreducible monic polynomial $f$ in $R[X]$ [\cite{8}, Theorem 28]. Since $M$ is a height two prime in a UFD R of Krull dimension 2, from the proof of Proposition \ref{prop} we have $M=\sqrt{(\Pi_1,\Pi_2)}$ where the only maximal ideals of $R$ containing the irreducible element $\Pi_1$  are $M, M_1, M_2,...,M_n$ and that $\Pi_2+z=1$ for some element $z$ in $\cap_{i=1}^{n}M_i.$
Next we set $g=\Pi_2+zf,$ then we have the ideal $(\Pi_1,g)$ is contained in $M^{*}.$ We claim that $M^{*}=\sqrt{(\Pi_1, g)}$ and that would contradict our assumtion. To prove this claim we first show that $M^{*}=j-\sqrt{(\Pi_1,g)}$. Suppose that $N^{*}$ is any maximal ideal of $R[X]$ containing the ideal $(\Pi_1,g).$ Then $N=N^{*}\cap R$ is a prime ideal of $R$ that contains the irreducible element $\Pi_1.$ Hence is either
of height 1 or is a height of 2. If $N$ is of height 1, then it is generated by an irreducible element and since it contains the irreducible element $\Pi_1$, we see that
$N=(\Pi_1)\subseteq M.$ Now localizing $R[X]$ at the maximal ideal $M$ of $R$ and noting that $z$ is not in $M$ we have the image of  $g=\Pi_2+zf$ in $N^{*}_{M}$ is a monic polynomial and therefore $R_{M}/N_{M}\subseteq R[X]_{M}/N^{*}_{M}$ is an integral extension. Since $R[X]_{M}/N^{*}_{M}$ is a field, $R_{M}/N_{M}$ is a field . Hence
$N_{M}=M_{M}$ in $R_{M}.$ But then $N=M$ in $R.$ That is $N$ can not be of height 1, and so is a maximal ideal that contains $\Pi_1 $ and hence is one of $M, M_{1}, M_{2},...,M_{n}.$
  If $N=M_i$ for some $i=1,2,...,n;$ then $N$ contains $z$ and so $N^{*}$ contains $\Pi_1, z, g=\Pi_2+zf$ and therefore it also contains $\Pi_2=g-zf.$ But then $N^{*}$ contains $1=\Pi_2+z,$ a contradiction. Therefore $N=M.$ Hence $N^{*}$ contains $\Pi_1, \Pi_2,   g=\Pi_2+zf$ and so it contains $g-\Pi_2=zf.$ Since $z$ is not in $N$ it is not in $N^{*}$ and therefore $f\in N^{*}.$ But then $N^{*}=(M,f)=M^{*}.$ That is $M^{*}=j-\sqrt{(\Pi_1, g)}.$

We now show that in fact $M^{*}$ is the only prime ideal of $R[X]$ containing the ideal $(\Pi_1,g)$. For let $P^{*}$ be any prime  ideal  of $R[X]$ containing  the ideal $(\Pi_1,g)$. Then clearly $M^{*}=j-\sqrt{P^{*}}$. Let  $P=P^{*}\cap R$ be the contraction of $P^{*}$ in $R$. Then $P$ contains the irreducible element $\Pi_1$ and hence is one of the prime ideals $(\Pi_1)$, $M, M_1, M_2,...,M_n$, as these are the only primes in $R$ containing  $\Pi_1$. Since each prime ideal $(\Pi_1)$, $M, M_1, M_2,...,M_n$, is a G-ideal [\cite{8}, Section 1-3, Theorem 19 and Theorem 146], P is a G-ideal. But this is so if and only if P is the contraction of a maximal ideal of the polynomial ring $R[X]$ [\cite{8}, Theorem 27]. Thus $P^{*}=M^{*}$ and therefore $M^{*}=\sqrt{(\Pi_1,g)}$.
This implies that $R[X]$ is of Krull dimension 2 which is in contradiction of being of Krull dimension 3.  Therefore such ting $R$ does not exist.
\end{proof}



\begin{corollary}
  Space curves are not always set theoretic complete intersection.
\end{corollary}

\begin{proof}
  Suppose otherwise that all space curves are set theoretic complete intersection. Then it follows from Theorem 1 of \cite{9} that all prime ideals of $K[Z,Y,X]$ (K is a field Z,Y,X are indeterminates) are radically perfect. Which in turn would imply that the polynomial ring $R[X]$ is radically finite, where $R$ is the localization of the ring $K[Z,Y]$ at the union of finitely many maximal ideals of height 2. (This works because localization commutes with the formation of radicals.) Which is not possible by Theorem \ref{4.3}.
\end{proof}

Finally, we note that there are non-principal Noetherian integral domains of Krull dimension one containing a field of characteristic zero over which the polynomial ring is not radically finite [ \cite{2}, Example1, page 6469]




\bibliographystyle{amsplain}

\end{document}